\documentclass[final]{siamltex}
\usepackage{amsmath,amssymb,graphicx,epsfig,euscript}
\title{On matrix-free computation of 2D unstable manifolds\thanks{This work was supported by Le Fonds Qu\'eb\'ecois de la Recherche sur la Nature et les Technologies grant nr. 2009-NC-125259 and by a Grant-in-Aid for Scientific Research of the Japan Society for Promotion of Science}}
\author{L. van Veen\thanks{Faculty of Science, University of Ontario
    Institute of Technology, 2000 Simcoe St. N., Oshawa, ON L1H 7K4, Canada ({\tt lennaert.vanveen@uoit.ca})} \and Genta Kawahara\thanks{Department of Mechanical Science and Bioengineering, Osaka University, 1-3 Machikaneyama, Toyonaka, Osaka 560-8531, Japan ({\tt kawahara@me.es.osaka-u.ac.jp})} \and Matsumura Atsushi\footnotemark[3]}

\begin{document}
\maketitle

\begin{abstract}
Recently, a flexible and stable algorithm was introduced for the computation of 2D unstable manifolds
of periodic solutions to systems of ordinary differential equations. The main idea of this approach
is to represent orbits in this manifold as the solutions of an appropriate boundary value problem.
The boundary value problem is under determined and a one parameter family of solutions can be found by means of
arclength continuation. This family of orbits covers a piece of the manifold. The quality of this covering
depends on the way the boundary value problem is discretised, as do the tractability and accuracy
of the computation. In this paper, we describe an implementation of the orbit continuation algorithm
which relies on multiple shooting and Newton-Krylov continuation. We
show that the 
number of time integrations necessary for each continuation step
scales only with the number of shooting intervals but not with the number of degrees
of freedom of the dynamical system. The number of shooting intervals is chosen based on
linear stability analysis to keep the conditioning of the boundary
value problem in check. We demonstrate our
algorithm with two test systems: a low-order model of shear flow and a well-resolved simulation
of turbulent plane Couette flow.
\end{abstract}

\begin{keywords} 
Unstable manifold, orbit continuation, Newton-Krylov continuation, shear turbulence.
\end{keywords}

\begin{AMS}
65P99, 37N10, 34K19, 65L10
\end{AMS}

\pagestyle{myheadings}
\thispagestyle{plain}
\markboth{L. van Veen {\it et al.}}{Matrix-free computation of 2D unstable manifolds}

\section{Introduction}

In recent years, an increasing number of algorithms from numerical
dynamical systems theory have become available for systems with many
degrees of freedom. These algorithms are designed for the 
computation and continuation of
equilibrium states, time-periodic solutions, invariant tori and
connecting orbits and are usually of the
prediction-correction variety. The prediction can be based simply on data
filtered from simulations or on extrapolation of a previously computed
part of the continuation curve. The correction step, however, involves
solving a large set of coupled nonlinear equations by an iterative method.
Because of its
quadratic convergence, the most desirable method here is Newton-Raphson
iteration. This method, in turn, requires the repeated solution of a
large linear system. It is not surprising, then, that the most
significant step forward in this field was the introduction of Krylov subspace
methods for solving the linear systems. 
The combination of
Newton-Raphson iteration with a Krylov subspace method for
prediction-correction methods is now referred
to as Newton-Krylov continuation.

Newton-Krylov continuation has been used extensively in the context of
fluid dynamics, where a large system of ordinary differential
equations (ODEs) results from discretisation of the Navier-Stokes
equation and possibly 
the continuity and energy equations.
Early examples include the computation of equilibrium states in Taylor
vortex flow by Edwards {\sl et al.} \cite{edwa} and the continuation of
time-periodic solutions by S\'anchez {\sl et al.} \cite{sanch2}. More
recently, the
algorithm has also been used for the computation of quasi-periodic
solutions \cite{sanch3}. 

Now that algorithms for the study of equilibria and (quasi) periodic
solutions are available, it is natural to consider their stable and unstable
manifolds and the global dynamical structures they represent. It is a
well-known fact from dynamical systems theory that these manifolds,
and the way they intersect in phase space, play an essential role in
such phenomena as the generation of chaos, boundary crises and
bursting behaviour. The tractability of manifold computation depends
critically on the manifold dimension. A one-dimensional (un)stable
manifold is just an integral curve of the ODEs and its computation is
simple. The computation of manifolds of dimension three and up would
be a formidable task. Even if we would design an algorithm which works
regardless of the dimension of the ambient space, the representation
of the manifold in a discrete data set and its interpretation would
present great difficulties. 

In the current paper we focus on the computation of two-dimensional
invariant manifolds. A number of algorithms has been designed for this
end, mostly for low-dimensional systems. An overview of methods can be
found in Krauskopf {\sl et al.} \cite{kraus2}. One method presented
there is particularly suitable for adaption to high dimensional
systems. This method is called {\em orbit continuation} \cite[Sec. 3]{kraus2} and is based
on a representation of the integral curves which fill the manifold as
solutions to an under determined boundary value problem (BVP). Using a
conventional prediction-correction method to approximate the
continuous, one-parameter family of solutions to this BVP, we construct an approximation to the manifold by finitely
many integral curves. Being originally designed for low-dimensional
systems, this method was implemented in the BVP
solver {\sf AUTO} \cite{auto}, which uses spectral collocation to represent the
integral curves and direct methods to solve the linear equations which
arise from Newton-Raphson iteration. In adapting the method to
high-dimensional systems, we discard spectral collocation in favour of
multiple shooting and implement Newton-Krylov continuation. Thus, we
strike a compromise between control over the covering of the manifold
on one hand and tractability of the algorithm on the other hand.

Two important points to consider when applying Newton-Krylov
continuation are the efficiency of the Krylov subspace method and the
extent to which the algorithm is parallelisable. These issues can be
closely connected. A good example is given by the continuation of
periodic solutions in Navier-Stokes flow. When using single shooting,
the Jacobian matrix of the Newton-Raphson iteration is very well
handled by Krylov subspace methods. This is because its spectrum is
strongly clustered \cite{sanch2}. If, however, we employ the highly
parallelisable multiple shooting algorithm, the eigenvalues spread out
over the complex plane and preconditioning  is necessary to ensure
linear speedup \cite{sanch}. We will show that Newton-Krylov orbit
continuation with multiple shooting does not require preconditioning. In
particular, we will show that the minimal dimension of the Krylov
subspace scales linearly with the number of shooting intervals but
does not depend on the number of ODEs.

We illustrate the algorithm with two examples. The first example is a
toy model of shear flow, originally introduced by Waleffe
\cite{walef}. In Sec. \ref{ex:wal} we show Newton-Krylov convergence
results and a comparison to {\sf AUTO} computations. The second example is
a well-resolved simulation of turbulent Couette flow, presented in
Sec. \ref{ex:cou}. In this case,
the high number of degrees of freedom in the simulation precludes the
use of spectral collocation and direct linear solvers. In either case,
we compute the unstable manifold of a periodic solution with a single
unstable multiplier. This is related to
an open issue in turbulent shear flow, namely the idea that
certain time-periodic solutions with a two-dimensional unstable
manifold play a key role in bursting behaviour \cite{eckh}, which
formed the original motivation for this work. Nevertheless, the
algorithm presented here has wider applicability. If an equilibrium
state has two unstable directions, we can simply adjust the left
boundary condition, as explained below. If we wish to compute a
two-dimensional stable manifold, we can either reverse time or
rephrase the BVP, reversing the role of the left
and right boundary conditions.

\section{Orbit continuation}

Let us denote the system of ODEs and the associated linearised system as
\begin{align}
\dot{\mathbf{x}} &= \mathbf{f}(\mathbf{x})\ \ \ \text{where} \ \ \mathbf{x}\in\mathbb{R}^n \label{ODE}\\
\dot{\mathbf{v}} &= D\mathbf{f} \,\mathbf{v} \label{ODElin}
\end{align}
and let the flow of system (\ref{ODE}) be denoted by $\phi(\mathbf{x},t)$. We assume there is a hyperbolic
periodic solution passing through the point $\bar{\mathbf{x}}$, i.e. 
$\bar{\mathbf{x}}=\phi(\bar{\mathbf{x}},0)=\phi(\bar{\mathbf{x}},P)$, where $P$ is the period.
We can write the solution to the linearised system (\ref{ODElin}) about this periodic solution as
$\mathbf{v}(P)=\mathbf{M}\,\mathbf{v}(0)$, where $\mathbf{M}$ is the monodromy matrix.
We assume that $\mathbf{M}$ has a single eigenvalue outside the unit circle in the complex plane, i.e.
$|\mu_n|\leq |\mu_{n-1}|\leq \ldots \leq \mu_{2}=1\leq \mu_{1}$. Let $\mathbf{u}_1$ be the corresponding eigenvector
based at $\bar{\mathbf{x}}$, i.e. $\mathbf{M}\mathbf{u}_1=\mu_1 \mathbf{u}_1$. 
Our aim is to compute a finite piece of the two dimensional unstable
manifold, tangent at $\bar{x}$ to the linear subspace spanned by $\mathbf{u}_1$ and $\mathbf{u}_2=\mathbf{f}(\bar{\mathbf{x}})$.

Orbits segments contained in this manifold, which we will denote by $\mathbf{\gamma}(t)$, approximately satisfy
the following boundary condition:
\begin{alignat}{2}
\mathbf{\gamma}(0) &= \bar{\mathbf{x}}+\epsilon \mathbf{u}_1 &\qquad &\text{(left boundary condition)}\nonumber\\
g(\mathbf{\gamma},T) &= c &\qquad &\text{(right boundary condition)}
\label{basicBVP}
\end{alignat}
where $g$ is a scalar function or functional. The choice of a suitable right boundary condition depends
on the problem and on the goal of the computation. Common choices,
which we will use throughout this paper, are
\begin{alignat*}{2}
g(\mathbf{\gamma},T)&=T &\qquad & \text{for orbits of fixed integration time}\\
g(\mathbf{\gamma},T)&=\int_{0}^{T}|\mathbf{f}(\mathbf{\gamma}(t))|\,dt &\qquad &\text{for orbits of
constant arc length}\\
g(\mathbf{\gamma},T)&=g(\mathbf{\gamma}(T)) &\qquad &\text{for orbits terminating on a Poincar\'e surface}
\end{alignat*}
The crucial observation is that this BVP is under determined by a single unknown. For a suitable
choice of the right boundary condition, there exists a continuous, one parameter family of
solutions which covers part of the manifold. Of course,
the left boundary condition is approximate and an error is introduced for finite $\epsilon$.
However, due to the exponential contraction transversal to the unstable manifold
this error does not increase along $\mathbf{\gamma}(t)$. 

In order to compute the two-dimensional unstable manifold of an
equilibrium state, all we need to do is replace the left boundary
condition. We fix a circle with a small radius in the subspace
spanned by the two unstable eigenvectors and demand that the initial
point lie on this circle. The free parameter is an angle \cite[Sec.
  3]{kraus2}.

To see that the BVP is under determined, it is instructive to think of it as a shooting problem.
Every orbit segment $\mathbf{\gamma}(t)$ is uniquely determined by its initial point
and integration time. Then, the $(n+2)$ unknowns of this BVP are $T$, $\epsilon$ and 
the components of $\mathbf{\gamma}(0)$, whereas conditions (\ref{basicBVP}) constitute $(n+1)$ equations. 

A standard method to numerically approximate the family of solutions is  arclength continuation. 
This is a prediction-correction
method. Let us write BVP (\ref{basicBVP}) compactly as
\begin{equation}
\mathbf{F}(\mathbf{z})=0 \qquad \text{where}\ \mathbf{z}^t=(\mathbf{\gamma}(0),T,\epsilon)
\label{compact}
\end{equation}
and let $\mathbf{T}$ denote the tangent to the family of solutions. Then the basic algorithm
can be summarised as follows:

\medskip

\framebox{
\begin{minipage}{.9\textwidth}
\newcounter{steps}
\begin{list}{\arabic{steps}.}{\usecounter{steps}\setlength{\itemsep}{4pt}\setlength{\leftmargin}{0.5\labelwidth}}
\item[] {\bf Single shooting arclength continuation of BVP (\ref{basicBVP})}
\item Find an initial solution by forward integration starting at 
$\mathbf{\gamma}(0)=\bar{\mathbf{x}}+\epsilon_0 \mathbf{u}_1$ and stopping when
$g(\mathbf{\gamma},T_0)=c$. Set $\mathbf{z}_0^t=(\mathbf{\gamma}(0),T_0,\epsilon_0)$ and
find $\mathbf{T}_0$.
\item Prediction: $\mathbf{z}^0_{i+1}=\mathbf{z}_i+\Delta s \,\mathbf{T}_i$.
\item Correction: solve
\begin{equation}
\mathcal{A} \,\delta\mathbf{z}^j=
\left( 
\begin{array}{c}
\  D\mathbf{F} \  \\
  \mathbf{T}_i^t  
\end{array} 
\right) \delta\mathbf{z}^j
=-\left(\begin{array}{c}\mathbf{F}(\mathbf{z}_{i+1}^j) \\ 0 \end{array} \right)
\label{basicNR}
\end{equation}
and update $\mathbf{z}_{i+1}^{j+1}=\mathbf{z}_{i+1}^j+\delta\mathbf{z}^j$
until a Newton-Raphson convergence criterion is met. Then set $\mathbf{z}_{i+1}=\mathbf{z}_{i+1}^j$.
\item Control step size $\Delta s$.
\item Repeat 2.-4. for $i=1,2, \ldots, i_{\rm max}$.
\end{list}
\end{minipage}
}

\medskip

In order to make the Newton-Raphson correction step unique we use the
condition that $\mathbf{T}_i \perp \delta \mathbf{z}^j_{i+1}$ in Eq. (\ref{basicNR}). 
This condition can always
be met for small enough step size $\Delta s$ and does not require extra computations.
As we will see below, step 4. is important because it allows us to put an upper bound on
the changes to unknowns and thus, indirectly, to control the covering of the manifold.

The main factor which decides if this scheme is feasible numerically is the condition of
the linear problem (\ref{basicNR}). In systems with sensitive dependence on initial conditions,
the linear problem gets increasingly harder to solve as we try to compute a larger part
of the manifold. This is most clearly seen if we select the right boundary condition that
the end points of the segments lie in a Poincar\'e surface of intersection. In that
case we have
\begin{equation}
\mathcal{A}=\left(\begin{array}{ccc}
\mathbb{I}_n & \mathbf{0} & -\mathbf{u}_1 \\
(\nabla g)^t D\phi & (\nabla g)^t \mathbf{f} & 0 \\\hline
 & \mathbf{T}^t & \end{array} \right)
\end{equation}
where $D\phi$ and $\mathbf{f}$ are evaluated at $\gamma(T)$. It can be
shown that both the 2-norm and the 
condition number of $\mathcal{A}$ are bound from below by
$\max_{j=1\ldots n}|(\nabla g)_i D\phi_{ij}|$ and we can expect $\|D\phi\|$ to grow
exponentially with the integration time $T$.

\section{The limitations of spectral collocation}

Without being rigorous, we can say that the optimal solution to the problem of
sensitive dependence on initial conditions lies in the use of spectral collocation.
Rather than to approach BVP (\ref{basicBVP}) as a single shooting problem, we can
represent the orbits $\mathbf{\gamma}(t)$ on a mesh using orthogonal basis functions.
This approach is used in the widely used software package {\sf AUTO} \cite{auto}.
A survey of results using spectral collocation for orbit continuation can be found in Krauskopf 
and Osinga \cite{kraus}. The main strength of this approach is that the set of unknowns
includes the complete set of collocation coefficients. When we control the
arclength step size, we control the change in shape of the entire orbit. Thus,
we can be reasonably certain that we miss no details of the geometry of the
manifold.

This control comes at a price, however. The total number of unknowns to solve
for in every step will be proportional to $N_{\rm m}\times N_{\rm c}\times n$,
where $N_{\rm m}$ is the number of mesh intervals and $N_{\rm c}$ is the number
of collocation points per mesh interval. The minimal number of mesh intervals
in turn depends on the number of mesh intervals necessary for resolving the periodic
solution, $N_{\rm m}^{(0)}$, and the unstable multiplier, $\mu_1$. If we assume that
the dynamics close to the periodic orbit is well described by the linearisation, 
the distance to the periodic orbit will grow as $\epsilon\mu_1^p$, where $p$ is the number
of times we integrate ``along'' the periodic orbit. We can think of $p$ as the number
of iterations of a Poincar\'e map. Thus, to resolve an orbit long enough to arrive
an $\mbox{O}(1)$ distance away from the periodic orbit, we will need about
$N_{\rm m}=-N_{\rm m}^{(0)}\ln(\epsilon)/\ln(\mu_1)$ mesh intervals. 

The result of this estimate depends very much on the details of the problem at hand.
The geometry of the vector field close to the periodic orbit will determine the trade-off
between the number of mesh intervals and the number of collocation points as well as the 
maximal value for $\epsilon$. We can, however, identify three possible settings in which 
the collocation approach is not tractable:
\begin{enumerate}
\item the number of degrees of freedom is large,
\item the unstable multiplier is very close to unity or
\item the periodic orbit has a complex shape.
\end{enumerate}
In this paper, we will consider two examples in which situations 1.\ and 2.\ arise. Situation
3.\ might be encountered, for instance, in multiple time scale systems
such as arise in neuro science.

Below, we will show that multiple shooting provides a good compromise between the accurate,
but costly, collocation approach and the cheap, but unstable, single shooting approach. In
particular, we will show that the three issues listed above will influence the computation time
only through the time it takes to perform sufficiently accurate forward time integrations.

\section{Multiple shooting}

In this approach, we represent the orbit on the unstable manifold as the concatenation of $k$ segments.
For each segment, we specify a scalar right boundary condition and, in addition, we have $(k-1)$ gluing
conditions to ensure the resulting orbit is continuous.

Let us define the vector of $N=(k-1) n+k+1$ unknowns as
\begin{equation}
\mathbf{z}=(\gamma^{(2)}(0),\ldots,\gamma^{(k)}(0),\,T^{(1)},\ldots,\,T^{(k)},\,\epsilon)
\nonumber 
\label{unknowns}
\end{equation}
and the set of $(N-1)$ nonlinear equations as
\begin{multline}
\mathbf{F}=\left(\gamma^{(2)}(0)-\gamma^{(1)}(T^{(1)}),\ldots,\gamma^{(k)}(0)-\gamma^{(k-1)}(T^{(1)}),\,g^{(1)}(\gamma^{(1)},T^{(1)})-c_1,\ldots,\right.\\
\left. g^{(k)}(\gamma^{(k)},\,T^{(k)})-c_k\right)
\label{MSBVP}
\end{multline}
We can rewrite the basic continuation algorithm for shooting on $k$ intervals as shown below. 
Instead of using direct methods, which require computation of the full matrix of derivatives, we can use a Krylov
subspace method. In particular, we will use GMRES \cite{saad}.

\medskip

\framebox{
\begin{minipage}{.9\textwidth}
\setcounter{steps}{1}
\begin{list}{\arabic{steps}.}{\usecounter{steps}\setlength{\itemsep}{4pt}\setlength{\leftmargin}{0.5\labelwidth}}
\item[] {\bf Multiple shooting Newton-Krylov continuation of BVP (\ref{MSBVP})}
\item[1.] Find an initial solution by forward integration starting from
$\mathbf{\gamma}(0)=\bar{\mathbf{x}}+\epsilon_0 \mathbf{u}_1$. Set $\mathbf{T}=(0,\ldots,0,1)^t$.
\item[2.] Prediction: $\mathbf{z}^0_{i+1}=\mathbf{z}_i+\Delta s \,\mathbf{T}_i$.
\item[3.] Correction: approximate the solution to
\begin{equation}
\mathcal{A} \,\delta\mathbf{z}^j=
\left( 
\begin{array}{c}
\  D\mathbf{F} \  \\
  \mathbf{T}_i^t  
\end{array} 
\right) \delta\mathbf{z}^j
=-\left(\begin{array}{c}\mathbf{F}(\mathbf{z}_{i+1}^j) \\ 0 \end{array} \right)
\label{basicNR2}
\end{equation}
by GMRES iterations up to tolerance $d$ and update 
$\mathbf{z}_{i+1}^{j+1}=\mathbf{z}_{i+1}^j+\delta\mathbf{z}^j$
until a Newton-Raphson convergence criterion is met. Then set $\mathbf{z}_{i+1}=\mathbf{z}_{i+1}^j$.
\item[4.] Control step size $\Delta s$.
\item[5.] Compute $\mathbf{T}$ by finite differences.
\item[6.] Repeat 2.-5. for $i=1,2, \ldots, i_{\rm max}$.
\end{list}
\end{minipage}
}

\medskip

We note that this algorithm is essentially the same as that employed by S\'anchez and Net \cite{sanch}.
Technically, the important distinction between the two algorithms is the structure of
$\mathcal{A}$. In the case of continuation of periodic orbits with multiple shooting,
its eigenvalues spread out in the complex plane and preconditioning is necessary to
ensure that the number of GMRES iterations in the innermost loop remains small. In the manifold
computation we will see that the number of GMRES iterations grows in proportion to the
number of shooting intervals but is independent of $n$ even without preconditioning.

The matrix $\mathcal{A}$ has the following structure
\begin{equation}
\mathcal{A}=\left(\begin{array}{cc} \mathcal{A}' & \mathbf{A} \\ \mathbf{B} & \mathbf{C} \end{array}\right)=
\left( \begin{array}{cccccc|c}
\mathbb{I} & & & & & & \\
-\mathbf{J}_2 & \mathbb{I} & & & & & \\
 & \ddots & \ddots & & & & \mathbf{A} \\
 & & -\mathbf{J}_{k-2} & \mathbb{I} & & & \\
 & & & -\mathbf{J}_{k-1} & \mathbb{I} & & \\\hline
 & & \mathbf{B} & & & & \mathbf{C} \\ 
\end{array}\right)
\label{structure}
\end{equation}
where $\mathbf{A}$ is a $(k-1)n \,\times \,(k+1)$ matrix of derivatives of the gluing conditions
with respect to the integration times and the small parameter, $\mathbf{B}$ is a $(k+1)\,\times\, (k-1)n$ matrix
of derivatives of the right boundary conditions with respect to the initial points
complemented by the first $(k-1) n$ components of $\mathbf{T}$
and $\mathbf{C}$ is a $(k+1)\,\times\,(k+1)$ matrix of derivatives of the right boundary
conditions with respect to the integration times and the small parameter complemented by the last $(k+1)$
components of $\mathbf{T}$. Their sparsity patterns are as follows:
\begin{alignat}{2}
A&=
\left(\!\begin{array}{cccccc}
\mathbf{a}_1 & \mathbf{0} & \cdots & \mathbf{0} & \mathbf{0} & \bar{\mathbf{a}} \\
\mathbf{0} & \mathbf{a}_2 & \cdots & \mathbf{0} & \mathbf{0} & \mathbf{0} \\
\vdots &\vdots & \ddots &\vdots &\vdots &  \vdots \\
\mathbf{0} & \mathbf{0} & \cdots & \mathbf{a}_{k-1} & \mathbf{0} & \mathbf{0}
\end{array}\!\right) 
\quad &
B&=
\left(\!\begin{array}{cccc}
\mathbf{0} &\mathbf{0} &\cdots &\mathbf{0}  \\
\mathbf{b}_2 &\mathbf{0} &\cdots &\mathbf{0} \\
\mathbf{0} & \mathbf{b}_2 & \cdots & \mathbf{0}\\
\vdots &\vdots & \ddots &  \vdots\\
\mathbf{0} &\mathbf{0} &\cdots& \mathbf{b}_k\\
\mathbf{t}_1 &\mathbf{t}_2 & \cdots& \mathbf{t}_{k-1}
\end{array}\!\right) \nonumber \\
C&=\left(\!\begin{array}{ccccc}
c_{11} & 0 & \cdots & 0 & \bar{c} \\
0 & c_{22} & \cdots & 0 & 0 \\
\vdots & \vdots & \ddots & \vdots & \vdots \\
0 & 0 &\cdots  & c_{kk} & 0 \\
t_1 & t_2 & \cdots  &t_k & t_{k+1}
\end{array}\!\right) 
\end{alignat}
Here, $\bar{\mathbf{a}}$ and each $\mathbf{a}_i$ is a column vector
with $n$ elements and each $\mathbf{b}_i$ and $\mathbf{t}_i$ is a row
vector with $n$ elements.
Because of the partitioning of $\mathcal{A}$, it is useful to
introduce the notation
$$
\mathbf{v}=(\mathbf{v}^{(1)},\mathbf{v}^{(2)})^t=(\mathbf{v}_1^{(1)},\ldots,\mathbf{v}_{k-1}^{(1)},v_{1}^{(2)},\ldots,v_{k+1}^{(2)})^t
$$
where each $\mathbf{v}_i^{(1)}$ is a column vector with $n$ entries
and each $v^{(2)}_i$ is a scalar. 

In order to establish the upper bound on the number of GMRES
iterations we need the following lemma.
\begin{lemma}\label{eigenv}
Matrix $\mathcal{A}$ has eigenvalue $\lambda_0=1$ with algebraic multiplicity at least $(k-1)(n-1)$ 
and geometric multiplicity at least $(n-1)$
\end{lemma}
\begin{proof}
We will show that a minimal set of $(n-1)$ linearly independent
eigenvectors exists for eigenvalue
$\lambda_0$ irrespective of the choice of boundary
conditions. We label these eigenvectors $\mathbf{w}_{1,i}$ ($i=1,\ldots,n-1$). Additional eigenvectors appear for each boundary condition
given by constant integration time and possibly at isolated points on
the continuation curve. Each of the eigenvectors in the minimal set has a preimage under
$(\mathcal{A}-\mathbb{I})^s$ for $s=1,\ldots,k-2$. We label the generalised eigenvectors
$\mathbf{w}_{s,i}$ so that
\begin{align}
(\mathcal{A}-\mathbb{I})\mathbf{w}_{s+1,i}&=\mathbf{w}_{s,i}\ \text{for}\ s=1,\ldots,k-2
  \label{evcondition1}\\
(\mathcal{A}-\mathbb{I})\mathbf{w}_{1,i}&=0 \label{evcondition2}
\end{align}
To prove this we
use induction on $s$. In the induction step we must prove the
existence of a generalised eigenvector $\mathbf{w}_{s+1,i}$ going on
the equation which $\mathbf{w}_{s,i}$ satisfies. For this end we use Fredholm's
alternative. Thus, we first examine the left null space of the
operator and introduce some notation that helps exploit the sparsity
patterns of the (generalised) eigenvectors.

Let $R_q$ denote the linear subspace of vectors with the
following sparsity pattern
$$
\mathbf{v}=(\mathbf{0},\ldots,\mathbf{0},\mathbf{v}_{k-q}^{(1)},\ldots,\mathbf{v}_{k-1}^{(1)},0,\ldots,0,v_{k+1-q}^{(2)},\ldots,v_{k}^{(2)},0)^t
$$
for $q=1,\ldots,k-1$. Also,
let $L$ denote the linear subspace of vectors with the
following sparsity pattern
$$
\mathbf{v}=(\mathbf{v}_1^{(1)},\mathbf{0},\ldots,\mathbf{0},v_{1}^{(2)},0,\ldots,0)^t
$$
and note that $R_1\subset R_2 \subset \ldots \subset R_{k-1}$ and $L\perp R_q$ for $q=1,\ldots,k-2$. 

First, we will assume that the functions $g^{(i)}$ which determine the right boundary conditions depend on the
initial condition $\gamma^{(i)}(0)$ on each shooting interval. This is the case for
right boundary conditions given by Poincar\'e planes of intersection
or constant arclength. The case in which one or more right boundary
conditions are given by constant integration time is discussed at the
end of the proof.

Under this condition, the null space of $(\mathcal{A}-\mathbb{I})^t$ is spanned by vectors $\mathbf{v}\in L$
for which
\begin{alignat}{2}
\mathbf{a}_1\cdot \mathbf{v}_1^{(1)}+(c_{11}-1) v_1^{(2)}&=0 \qquad &
\bar{\mathbf{a}}\cdot \mathbf{v}_1^{(1)}+\bar{c}v_{1}^{(2)}&=0
\end{alignat}
Note, that these conditions are linearly independent because
$\mathbf{a}_1$ and $\bar{\mathbf{a}}$ are transversal. They correspond
to variations of the final point of the first shooting segment
resulting from variation in $T^{(1)}$ and $\epsilon$. These variations
are the image under the nonsingular matrix $\mathbf{J}_1$ of
$\mathbf{f}(\bar{\mathbf{x}}+\epsilon \mathbf{u}_1)$ and
$\mathbf{u}_1$, which are transversal for sufficiently small $\epsilon$
because in the limit of $\epsilon\downarrow 0$ they are eigenvectors of
the monodromy matrix of the periodic orbit for distinct eigenvalues.
Thus, we have $\text{dim}(L)-2=(n-1)$ left null vectors contained in $L$.

We now construct $(n-1)$ eigenvectors contained in $R_1$. Let
\begin{equation}
R_1\ni \mathbf{w}_{1,i}=(\mathbf{0},\ldots,\mathbf{v}_{k-1}^{(1)},0,\ldots,v_{k}^{(2)},0)^t
\end{equation}
then Eq. (\ref{evcondition2}) gives
\begin{alignat}{2}
\mathbf{b}_k\cdot \mathbf{v}_{k-1}^{(1)}+(c_{kk}-1) v_{k}^{(2)}&=0 &\qquad  \mathbf{t}_{k-1}\cdot \mathbf{v}_{k-1}^{(1)}+t_{k}
v_{k}^{(2)}&=0
\label{ev0}
\end{alignat}
so that the number of eigenvectors is at least $\text{dim}(R_1)-2=n-1$.
An extra eigenvector may exist at isolated points on the continuation
curve where the two condition are linearly dependent. We defer a discussion of such special
points to the end of this proof.
This concludes the proof for $k=2$. 

For $k=3$ it suffices to note that, since $R_1\perp L$,
Eq. (\ref{evcondition1}) has a solution to each $i=1,\ldots,n-1$ by
Fredholm's alternative. In this case there are $(n-1)$ linearly
independent eigenvectors and $(n-1)$ linearly independent generalised eigenvectors.

For $k>3$ we use induction on $s$. Suppose that a generalised eigenvector $\mathbf{w}_{s,i}\in R_s$
exists for $s<k-1$. Then it has a preimage under $(\mathcal{A}-\mathbb{I})$ by Fredholm's alternative.
We need to show that this preimage has a non-empty intersection with $R_{s+1}$.
Let
$\mathbf{w}_{s+1,i}=(\mathbf{v}_1^{(1)},\ldots,\mathbf{v}_{k-1}^{(1)},v_{1}^{(2)},\ldots,v_{k+1}^{(2)})^t$.
The condition that
$(\mathcal{A}-\mathbb{I})\mathbf{w}_{s+1,i}\in R_{s}$ gives
\begin{alignat}{2}
v_{1}^{(2)}\mathbf{a}_1+v_{k+1}^{(2)}\bar{\mathbf{a}}&=0 & \qquad
(c_{11}-1)v_{1}^{(2)}+\bar{c}v_{k+1}^{(2)}&=0
\label{c0}
\end{alignat}
from which we find $v_{1}^{(2)}=v_{k+1}^{(2)}=0$. Next, we have 
\begin{align}
-\mathbf{J}_i\mathbf{v}_{i-1}^{(1)}+v_{i}^{(2)}\mathbf{a}_i&=-\mathbf{J}_i\mathbf{v}_{i-1}^{(1)}-v_{i}^{(2)}\mathbf{f}(\gamma^{(i)}(T^{(i)}))= 0 \label{c1}\\
\mathbf{b}_i\cdot\mathbf{v}_{i-1}^{(1)}+(c_{ii}-1)v_{i}^{(2)}&=0 \label{c2}
\end{align}
for $i=2,\ldots,k-s-1$. The most general solution of Eq. (\ref{c1}) is
$\mathbf{v}_{i-1}^{(1)}=\alpha \mathbf{f}(\gamma^{(i)}(0))$,
$v_i^{(2)}=-\alpha$. The second equation depends on the $i^{\rm th}$
  boundary condition.
If the boundary condition is a Poincar\'e plane of
intersection, then $\mathbf{b}_i=(\nabla g_i)^t\mathbf{J}_i$ and
$c_{ii}=(\nabla g_i)^t \mathbf{f}(\gamma^{(i)}(T^{(i)}))=-(\nabla g_i)^t \mathbf{a}_i$
and it follows that $\mathbf{w}_{i-1}=\mathbf{0}$ and $w_{i}^{(1)}=0$.
If the right boundary condition is given by constant arc length, a
solution with $\alpha\neq 0$ exists only if
$\|\mathbf{f}(\gamma^{(i)}(0))\|=1$. At isolated points on the
continuation curve where this holds, an extra eigenvector for
$\lambda_0$ exists. This situation is discussed at the end of the proof.
In the generic case, the conclusion is that any vector in the preimage
of $\mathbf{w}_{s,i}$ is contained in $R_{s+1}$. 
This completes the proof, with the exception of the special cases
discussed below.

For each right boundary condition given by constant integration time,
an extra eigenvector for eigenvalue $\lambda_0$ exists. In particular,
if $g^{(i)}=T^{(i)}$ for $i>1$ then there is a right eigenvector
$\mathbf{z}\in R_{k+1-i}$ and a left eigenvector $\mathbf{e}_{(k-1)n+i}$,
i.e. the $([k-1]n+i)^{\rm th}$ unit direction vector. It is
straightforward to see that $\mathbf{e}^t_{(k-1)n+i}\mathbf{z}\neq 0$
so that this eigenvector does not have any generalised eigenvectors
associated with it. The proof by induction for the other (generalised)
eigenvectors still holds. The only difference is that if we consider
Eq. (\ref{evcondition1}) in the induction step we find that the
preimage of $\mathbf{w}_{s,i}$ is not contained in $R_{s+1}$. However,
its intersection with $R_{s+1}$ is non-empty. If $g^{(1)}=T^{(1)}$
then no additional eigenvector exists. Instead, one of the
eigenvectors in the minimal set has an additional preimage.

In the exceptional cases that Eqs. (\ref{ev0}) are linearly dependent
or that Eqs. (\ref{c1})-(\ref{c2}) allow for a nonzero solution, 
an additional eigenvector exists. These cases are treated just like the
appearance of an additional eigenvector discussed above. Again, it is
straightforward to show that the preimage of each generalised
eigenvector $\mathbf{w}_{s,i}$ intersects the right subspace $R_{s+1}$.
\end{proof}

Of course, the discussion of the special cases which arise only at
isolated points on the continuation curve is somewhat academic, as we
will compute a discrete approximation to this curve. It is good to
know, however, that no drastic changes in the eigenspectrum of
$\mathcal{A}$ occur. As we will see in the following proposition, this
guarantees that a global maximum for the number of GMRES iterations
can be computed a priori.

\begin{proposition}\label{max_dim}
Assume that all eigenvalues of $\mathcal{A}$ other than $\lambda_0=1$
are simple. Then
the number of GMRES iterations necessary to solve (\ref{basicNR2}) is
at most $(3k-1)$
with exact arithmetic.
\end{proposition}
\begin{proof}
First, assume that the right boundary conditions depend on the initial
conditions on each shooting interval. Then,
by Lemma \ref{eigenv}, $\mathcal{A}$ has the Jordan normal form
$$
\mathcal{A}_{\rm J}=\mathbf{Q}\mathcal{A}\mathbf{Q}^{-1}=\left(\begin{array}{cccccc} \mathbf{M} & & & & & \\
 & \ddots & & & & \\
 & & \mathbf{M} & & & \\
 & & & \lambda_1 & & \\
 & & & & \ddots & \\
 & & & & & \lambda_{2k}\end{array}\right)
$$
where $\mathbf{M}$ is a Jordan block of dimension $(k-1)$ such that $(\mathbf{M}-\mathbb{I})^{k-1}=\mathbf{0}$.
After $p$ GMRES iterations with the initial vector $\mathbf{w}$, the residue is bound from above by
$$
\min_{P_p(0)=1} \|P_p(\mathcal{A})\mathbf{w}\|=\min_{P_p(0)=1} \|\mathbf{Q}^{-1}P_p(\mathcal{A}_{\rm J})\mathbf{Q}\mathbf{w}\|
\leq \kappa(\mathbf{Q}) \|\mathbf{w}\| \min_{P_p(0)=1}\|P_p(\mathcal{A}_{\rm J})\|
$$
where the minimum is taken over all polynomials of order $p$ which satisfy $P_p(\mathbf{0})=\mathbb{I}$.
Then we have
$$
P_{3k-1}(\mathcal{A}_{\rm J})=\frac{(-1)^{k-1}}{\Pi_{i=1}^{2k}\lambda_i}(\mathcal{A}_{\rm J}-\mathbb{I})^{k-1} \Pi_{i=1}^{2k} 
(\mathcal{A}_{\rm J}-\lambda_i \mathbb{I})=\mathbf{0}
\label{proof}
$$
As $P_{3k-1}$ is a polynomial of order $3k-1$ this proves the
proposition.

For every right boundary condition independent of the initial
condition one of the simple eigenvalues is equal to $\lambda_0$. Thus,
for each of these we can omit one GMRES iteration. The only exception
is the first shooting interval. If the first boundary condition is
independent of the initial condition, in that case meaning independent
of variation of
$\epsilon$, we have one less simple eigenvalue different from
$\lambda_0$, but one of the Jordan blocks is of dimension $k$, so that
the minimal number of iterations is $3k-1$.
\end{proof}

\section{Implementation and parallelism}

The Newton-Krylov continuation algorithm is easy to parallelise. For
each iteration of GMRES we need to compute the matrix-vector product
$\mathcal{A}\mathbf{v}$ for some given perturbation vector
$\mathbf{v}$. The constituents of this product are found by
integrating the extended system (\ref{ODE}-\ref{ODElin}) on each of
the shooting intervals. These integrations are all independent and can
be executed on different CPUs. The matrix-vector product is then
formed by the root process. This step involves only $\mbox{O}(N)$
elementary operations and the communication of vectors of length
$n$. Consequently, the overhead is very small and the examples below
show nearly $100\%$ efficiency of the parallelisation.

In the first example, the set of ODEs is only of dimension $17$ and
the essential loop of the code is easily parallelised with {\sf openMP}. 
In the second example, each integration is done by a pseudo-spectral
Navier-Stokes simulation code~\cite{itano}.
In this case, distributed memory {\sf MPI}
parallelisation is employed.

We need to make two remarks about the stability of the
algorithm. Firstly, it is more stable to use a logarithmic scale for
the small parameter, i.e. $\delta=\ln(\epsilon)$. If, in addition, we
normalise the integration times by the period of the unstable periodic
orbit, the dependent variables in the continuation are of comparable
size - assuming that the phase space variables are suitably scaled.
Secondly, the matrix $\mathcal{A}$ is ill-conditioned. The condition
number can be expected to increase exponentially with the integration
time on each shooting interval. As a consequence, the
orthogonalisation of the basis of the Krylov subspace can be
unstable. This problem is largely solved by using a $QR$-decomposition
based on Householder transformations.

The third remark concerns the left boundary condition in the BVP (\ref{basicBVP}). In order
to improve the accuracy of the manifold computation we can add a
second order term to its local approximation. Using a higher order
approximation, we can generally allow for larger values of $\epsilon$ in the continuation.
To compute the second order term, we
need to integrate the second order variational equations along the
periodic orbit. This is not normally feasible for high-dimensional
systems.

Finally, if the system under consideration has a strong dependence on initial conditions, we 
must start the continuation with a short orbit obtained by forward time integration.
We can choose a Poincar\'e plane which intersects with the periodic orbit for a right boundary
condition and let the integration time increase in the arclength continuation, adding shooting
intervals when necessary. When the computed orbit is long enough,
we can switch to a different boundary condition. The flexibility of the algorithm to select 
different parts of the manifold for computation by selecting different boundary conditions is 
one of its main strengths.

\section{Example computations}

In the following section we will describe test computations with a toy
model as well as a full-fledged simulation of turbulent shear flow. In
both models, there exists a periodic solution which seems to organise
the phenomenon of bursting. In a bursting flow, we see turbulent
episodes, during which the fluid motion is highly complex, interspersed
with nearly laminar episodes, during which the motion is
smooth. The periodic solution of interest has  a single unstable
multiplier and as a consequence its stable manifold separates the
phase space~\cite{kawa05}.
Special solutions like this are sometimes called {\em
  edge states} \cite{eckh}. In the simplest explanation of bursting, the phase point
is attracted to the edge state during the laminarisation, then moves
away from the edge state along a two-dimensional unstable manifold
during the bursting phase. Complete laminarisation will not occur
because the domain of attraction of the laminar flow is bounded by the
stable manifold of the edge state. Computation of the unstable
manifold we will give information about the transition from a
near-laminar configuration to a turbulent one. 

\subsection{A low-order model of shear flow: weak
  instability\label{ex:wal}}

For the first example computation we use a model for shear flow originally introduced
by Waleffe \cite{walef}. The model is obtained as a Galerkin
truncation of the Navier-Stokes equation for an incompressible fluid trapped between two
infinite, parallel plates with free slip boundary conditions. Energy
is input by a sinusoidal body force and Fourier modes are used in all
directions. Waleffe formulated this model to demonstrate the
regeneration cycle in shear flows, i.e. the repeated formation and
breakdown of stream wise vortices and low velocity
streaks. Accordingly, the modes retained in the Galerkin truncation
were chosen to have the spatial symmetries of stream wise vortices,
streaks and streak instabilities. In the original paper \cite{walef},
the maximal wave number was set to $2$ in the wall-normal direction
and $1$ in the stream wise and span wise directions. Here, we consider maximal
wave numbers $3$ and $1$, respectively, which leads to a set of $17$
nonlinear, coupled ODEs.

Obviously, such a severe truncation can only be regarded as a toy
model of shear flow. Remarkably though, the low-order model has many
of the qualitative traits that make shear flow so challenging from a
dynamical systems point of view. First of all, there exists a linearly
stable laminar solution for all Reynolds number. Secondly, for high
Reynolds numbers, solutions to the ODE typically show chaotic bursts
interspersed with smooth behaviour. The model introduced by Waleffe has often been
used as a test case for new ideas and algorithms for parsing turbulent
bursting, see e.g. Moehlis {\sl et al.} \cite{moeh} and references therein.

\begin{figure}
\begin{center}
\includegraphics[width=0.9\textwidth]{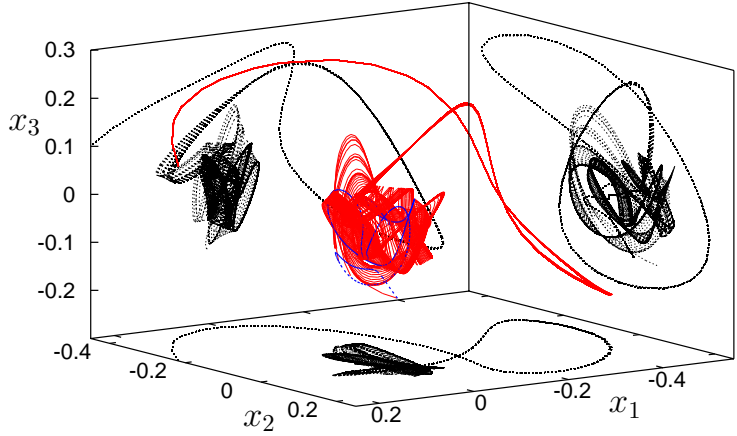}
\end{center}
\caption{A piece of the unstable manifold of the periodic orbit at the
edge of chaos in a toy model of shear flow. On the axes are the first
three Fourier coefficients. The blue curve is the intersection of this
piece of manifold with the plane $x_1=0$, the rightmost
boundary condition. The periodic orbit is located near the equilibrium 
$x_1=1$, $x_i=0$ for $i=2,\ldots n$ and is not shown. The computed orbits 
remain close until they reach the chaotic region, where they flare out
exponentially.
The corresponding continuation diagram is shown in
Fig. \ref{cont_curve_wal}.}
\label{man_wal}
\end{figure}

In our model we fix the stream wise to wall normal aspect ratio to
$L/H=2.76$
and the span wise to wall normal aspect ration to $W/H=1.88$,
corresponding to the minimal flow unit of plane Couette flow~\cite{hamilton}.
A bifurcation analysis reveals that at $Re=109$ a saddle type and a
stable periodic solution are created in a saddle-node bifurcation. The
stable orbit loses stability in a torus bifurcation at $Re=256$. The
saddle type orbit has a single unstable multiplier for any
$Re>109$. This orbit is a small perturbation of the laminar state and
can be considered an edge state. In the following, we fix $Re=667$ and
compute the unstable manifold of this edge state.
A difficulty in this computation is that the unstable multiplier is
close to unity. In the example computation it is $\mu_1=1.055$.

\begin{figure}
\begin{center}
\includegraphics[width=0.7\textwidth]{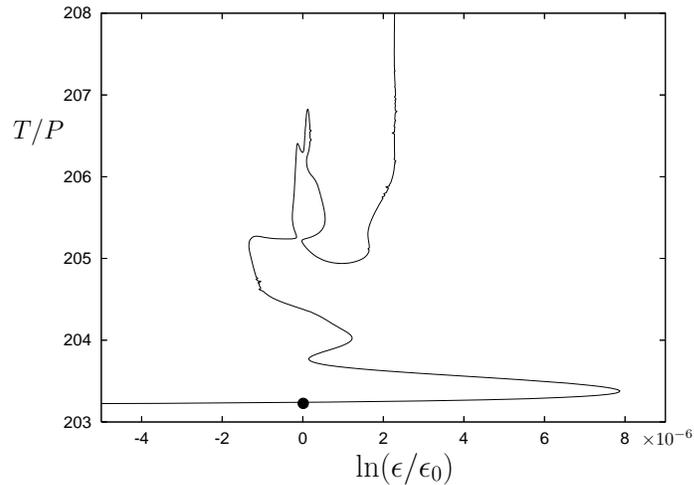}
\end{center}
\caption{The continuation curve corresponding to the piece of manifold
  shown in Fig. \ref{man_wal}. On the horizontal axis the small
  parameter which fixes the left boundary condition, on the vertical
  axis the total integration time. In this continuation, there were
  five shooting intervals and the boundary conditions were given by
  fixed integration time for the first interval and a
  Poincar\'e section $x_1=\text{constant}$ for the others.}
\label{cont_curve_wal}
\end{figure}

Fig. \ref{man_wal} shows a piece of the unstable manifold, computed on
five shooting intervals with a quadratic local approximation. The
right boundary conditions were fixed integration time for the first
interval and a Poincar\'e section for the other intervals. Near
the end of the computed orbits there is very strong dependence on
initial conditions and the geometry of the manifold is quite
complex. The variations in $\epsilon$ are extremely small.
Fig. \ref{cont_curve_wal} shows the corresponding continuation
curve. In this graph, a fold point corresponds to an orbit which is
tangent to one of the Poincar\'e planes of intersection. A detailed
account of the geometry of the manifold in the vicinity of such points
can be found in Lee {\sl et al.} \cite{lee}.

\begin{figure}
\begin{picture}(200,200)
\put(0,0){\epsfig{file=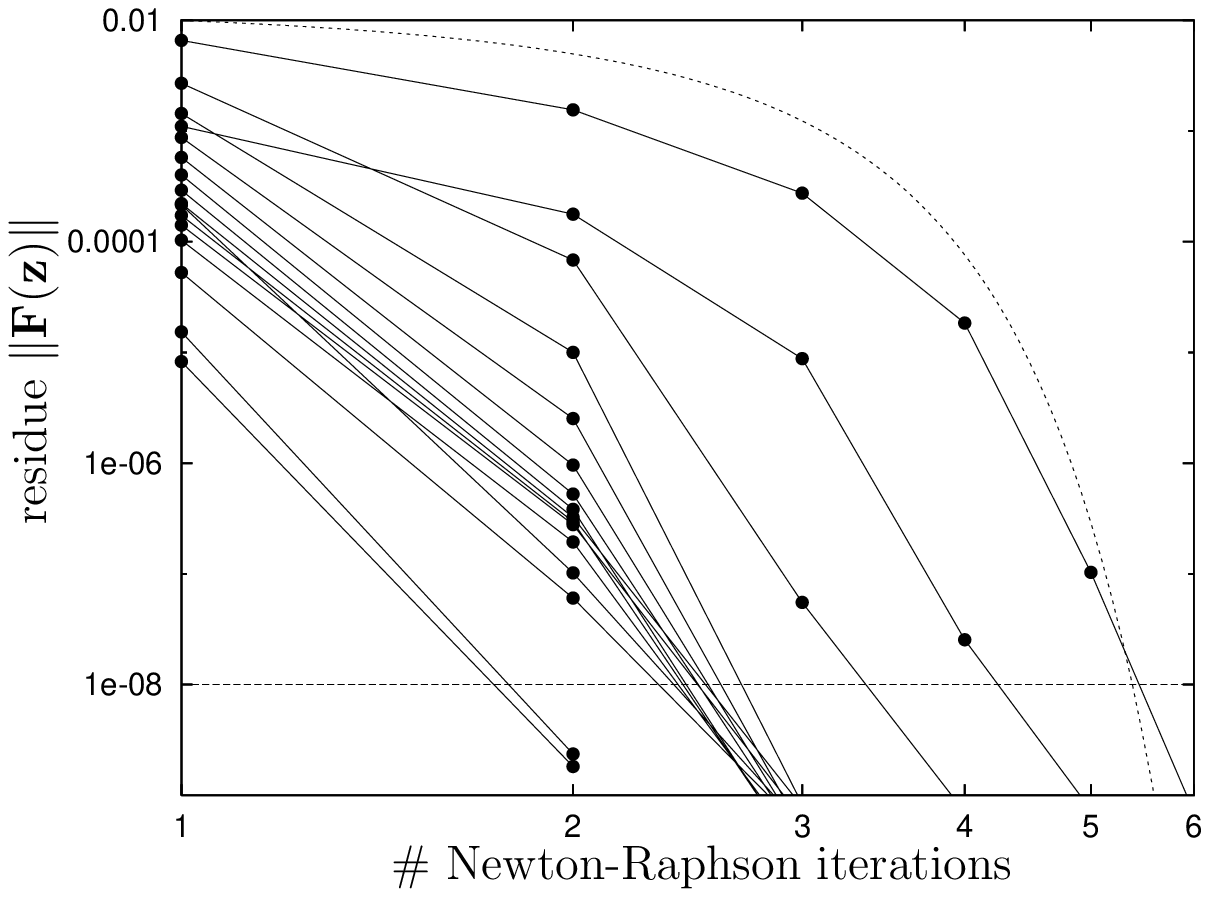,width=180pt}}
\put(185,-3){\epsfig{file=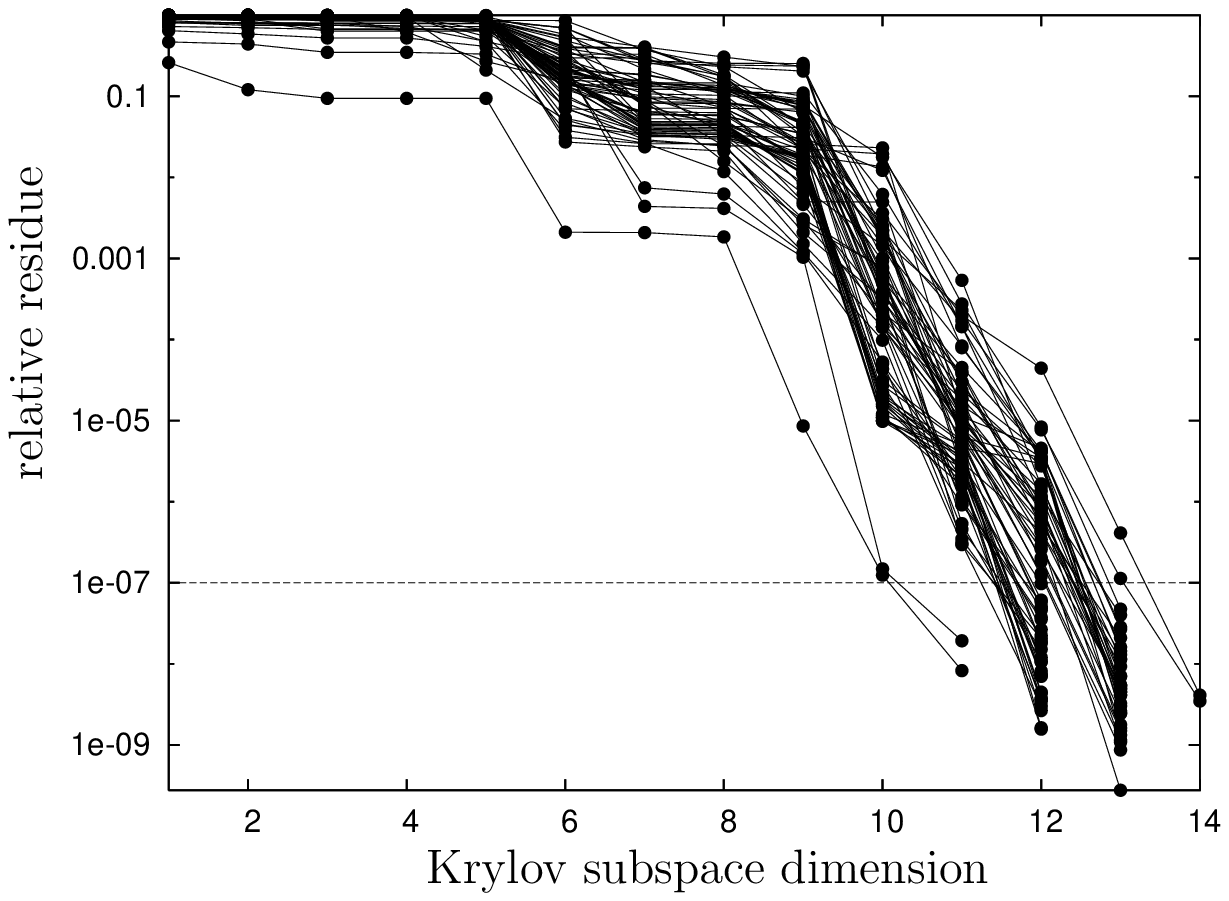,width=180pt}}
\end{picture}
\caption{Convergence results for the computations in the low-order
  model of shear flow. Left: the residue $\|\mathbf{F}(\mathbf{z})\|$ of the correction step versus the number of
Newton-Krylov iterations. The dotted line denotes quadratic
convergence. Iterations were stopped when the residue was less than
$10^{-8}$. Right: the GMRES residue, normalised by
$\|\mathbf{F}(\mathbf{z})\|$, as a function of the Krylov subspace
dimension. These results were obtained for the continuation shown in
Fig. \ref{cont_curve_wal}, with five shooting intervals. By
proposition \ref{max_dim} the maximal subspace dimension is $14$. }
\label{conv_wal}
\end{figure}

Fig. \ref{conv_wal} shows the convergence of the Newton-Krylov
iterations. Clearly, the convergence of the Newton iterations is super
linear and the number of GMRES iterations satisfies Proposition \ref{max_dim}.

In order to compare the multiple shooting algorithm to spectral
collocation, we implemented the basic BVP
(\ref{basicBVP}) in {\sf AUTO} \cite{auto}. The latest version of this
software package allows for thread-parallelisation of the linear
solver so that we can compare the performance of the methods for
different numbers of processors. Fig. \ref{wtime_cpu_wal} shows the
wall time for computing a fixed piece of the continuation curve. For
the multiple shooting we used two different strategies. First, we
fixed the number of shooting intervals to three. In that case, the
computation time decreases approximately by factors of $2/3$ and $1/3$
as we increase the number of CPUs to two and three. After that adding
CPUs has no effect. This result is shown with circles. Then,
took the number of shooting intervals to be equal to the number of
CPUs. In this case, we cannot predict whether the wall time will
decrease because the upper bound on the number of GMRES iterations
increases linearly with the number of shooting intervals. In practice,
we see that the wall time does decrease, albeit slower than linear.
This result is shown with squares. {\sf AUTO} results are shown with triangles.
Clearly, wall time taken by {\sf AUTO} scales nearly linearly over a
large range of numbers of CPUs. Nevertheless, the shooting method is
faster up to six CPUs.
\begin{figure}
\begin{center}
\includegraphics[width=0.7\textwidth]{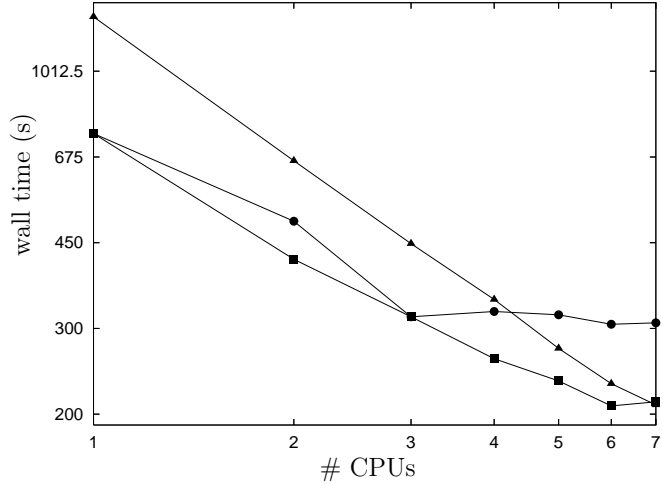}
\end{center}
\caption{Wall time versus the number of CPUs for a representative part of the continuation presented in Fig. \ref{cont_curve_wal}. The dots have been connected by lines for easy comparison.
Triangles represent AUTO results, circles represent shooting on three intervals and squares represent shooting with a variable number of intervals. For precise
parameters see text. The scaling of the computation time is linear for the spectral collocation used by AUTO. For shooting on few intervals the wall time is essentially set by
the longest integration time. }
\label{wtime_cpu_wal}
\end{figure}

\subsection{Transition in plane Couette flow: many degrees of freedom\label{ex:cou}}

Next we compute numerically the unstable manifold of the gentle periodic
orbit in a full plane Couette system~\cite{kawa}.
This periodic orbit has been obtained at Reynolds number $Re=400$
for the minimal periodic box
$(L/H,W/H)=(2.76,1.88)$~\cite{hamilton}.
The linear stability analysis of the periodic orbit has shown that
there is only one unstable multiplier, implying that the periodic orbit and
its stable manifold form the basin boundary between laminar and turbulent
attractors~\cite{kawa05}.
Transitions starting with a disturbance of the laminar flow just beyond a critical amplitude
can be described in terms of the unstable manifold of the periodic orbit.

In the computation of the unstable manifold of the gentle periodic orbit
we perform direct numerical simulations for the imcompressible Navier--Stokes
equation by use of a pseudo-spectral code~\cite{itano}.
In this code the streamwise volume flux
and the spanwise mean pressure gradient are set to be zero.
The dealiased Fourier expansions are employed in the streamwise and
spanwise directions, and the modified Chebyshev-polynomial expansion in the wall-normal
direction.
Nonlinear terms in the Navier--Stokes equation
are computed on $8448$ ($=16\times 33\times 16$
in the streamwise, wall-normal and spanwise direction) grid points.
The spatial symmetries
observed in a turbulent state of the minimal periodic box
are imposed on the periodic orbit~\cite{kawa}.
The dealiased symmetric flow field satisfying noslip and impermeable boundary
conditions has $2477$ degrees of freedom.
Time integration of the equation
is performed by using the explicit Adams--Bashforth method for the
nonlinear terms and the implicit Crank--Nicholson scheme for the
viscous terms.

Fig. \ref{man_cou} shows a piece of the unstable manifold projected on
energy input rate, energy dissipation rate and the energy contained in
the velocity field after subtracting the average velocity in the
stream wise direction.  In this computation three shooting intervals
were used. The first boundary condition is given by constant
integration time and the second by a Poincar\'e plane of intersection
on which the sum of energy input and dissipation rate is constant.
The rightmost boundary condition is given by
constant arc length. At the edge of the computed piece of manifold, the values of the energy input
and dissipation rate are comparable to their time mean value in turbulent flow.

In temporal evolution along the unstable manifold
the flow has been found to exhibit the same spatiotemporal
behaviour as observed in the transition to turbulence
in minimal plane Couette flow.
Low-velocity streaks develop with an oscillatory bend in the spanwise direction.
During this process the spanwise bend of the streak is enhanced to generate
a pair of staggered counter-rotating streamwise vortices in the flanks of
the low-velocity streak.
The generated streamwise vortices appear to be significant around the intersection
$e+\EuScript{E}=6$ which is comparable to the value corresponding to
a turbulent state.
\begin{figure}
\begin{center}
\includegraphics[width=0.8\textwidth]{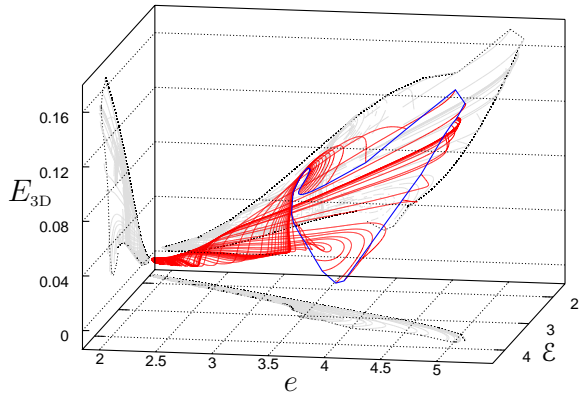}
\end{center}
\caption{A piece of the unstable manifolds of the quiescent periodic
  orbit in plane Couette flow, projected on energy input rate $e$,
  energy dissipation rate $\EuScript{E}$ and $E_{\rm \scriptstyle
    3D}$, the energy contained in the
  velocity field after subtraction of the average in the stream wise direction. The energy input and dissipation rate have been
  normalised by their value in laminar flow. This figure can be
  compared to Fig. 5 of Kawahara and Kida \cite{kawa}, in which a
  single orbit contained in the unstable manifold is
  displayed. Although the computed piece of manifold stretches from
  the near-laminar to the turbulent region in phase space, it looks
  like a cylinder and the intersection with a Poincar\'e plane is a
  simple closed curve.
The corresponding continuation diagram is shown in
Fig. \ref{cont_curve_cou}.}
\label{man_cou}
\end{figure}

Fig. \ref{cont_curve_cou} shows the continuation curve corresponding
to the piece of manifold of Fig. \ref{man_cou}. There are two points
where orbits are tangent to the Poincar\'e plane of intersection. The
first and the last point in this diagram correspond to orbits which
coincide in the projection of Fig. \ref{man_cou} but in phase space
are related by a discrete symmetry composed of a reflection in the
span wise direction, combined with a shift in the stream wise
direction. This explains why the two orbits differ only by half the
period of the quiescent periodic orbit.
\begin{figure}
\begin{center}
\includegraphics[width=0.7\textwidth]{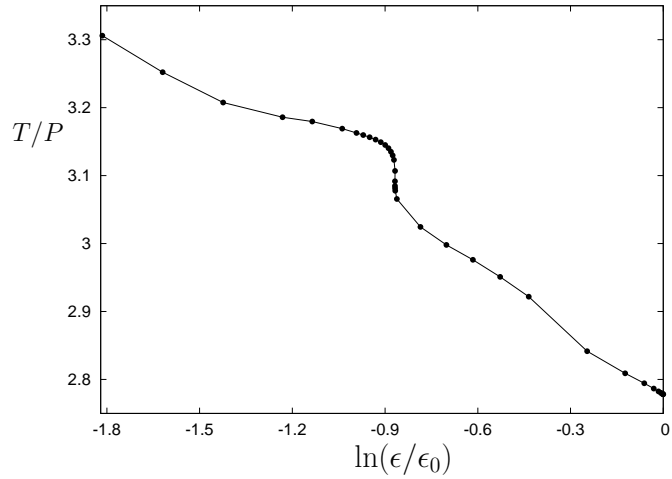}
\end{center}
\caption{The continuation curve corresponding to the piece of manifold
  shown in Fig. \ref{man_cou}. On the horizontal axis the small
  parameter which fixes the left boundary condition (\ref{basicBVP}), on the vertical
  axis the total integration time, normalised by the period of the
  quiescent periodic orbit. In this continuation, there were
  three shooting intervals and the boundary conditions were given by
  fixed integration time, the Poincar\'e section $e+\EuScript{E}=6.5$ and
  fixed arc length, respectively. The first and the last point in this
  continuation correspond to orbits related be a discrete spatial
  symmetry which leaves the energy invariant. 
  }
\label{cont_curve_cou}
\end{figure}

The convergence results for our computation in Couette flow are shown
in Fig. \ref{conv_cou}. Like for the low-order model, the convergence
of the Newton-Krylov iteration is super linear and the dimension of
the Krylov subspace satisfies proposition \ref{max_dim}.
\begin{figure}
\begin{picture}(200,200)
\put(0,0){\epsfig{file=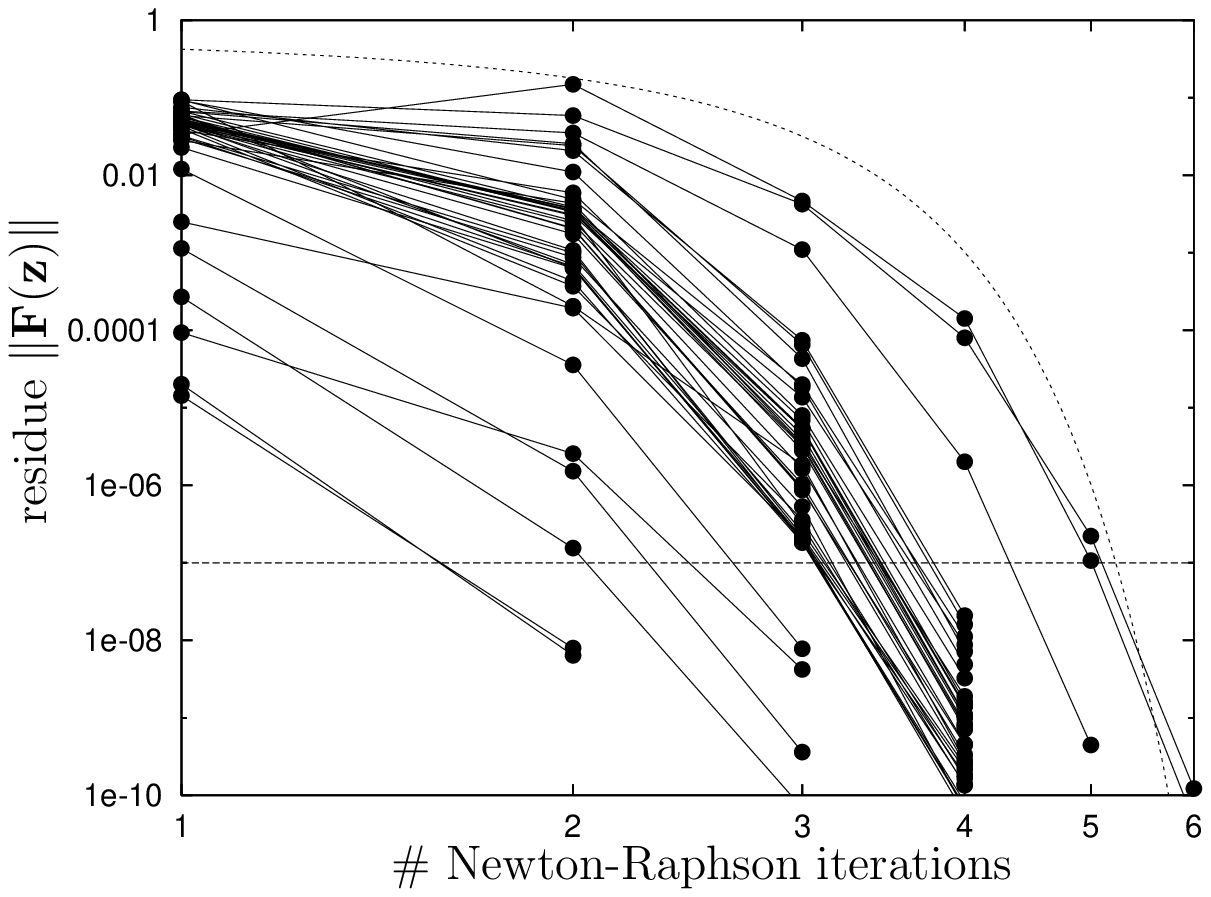,width=180pt}}
\put(185,-3){\epsfig{file=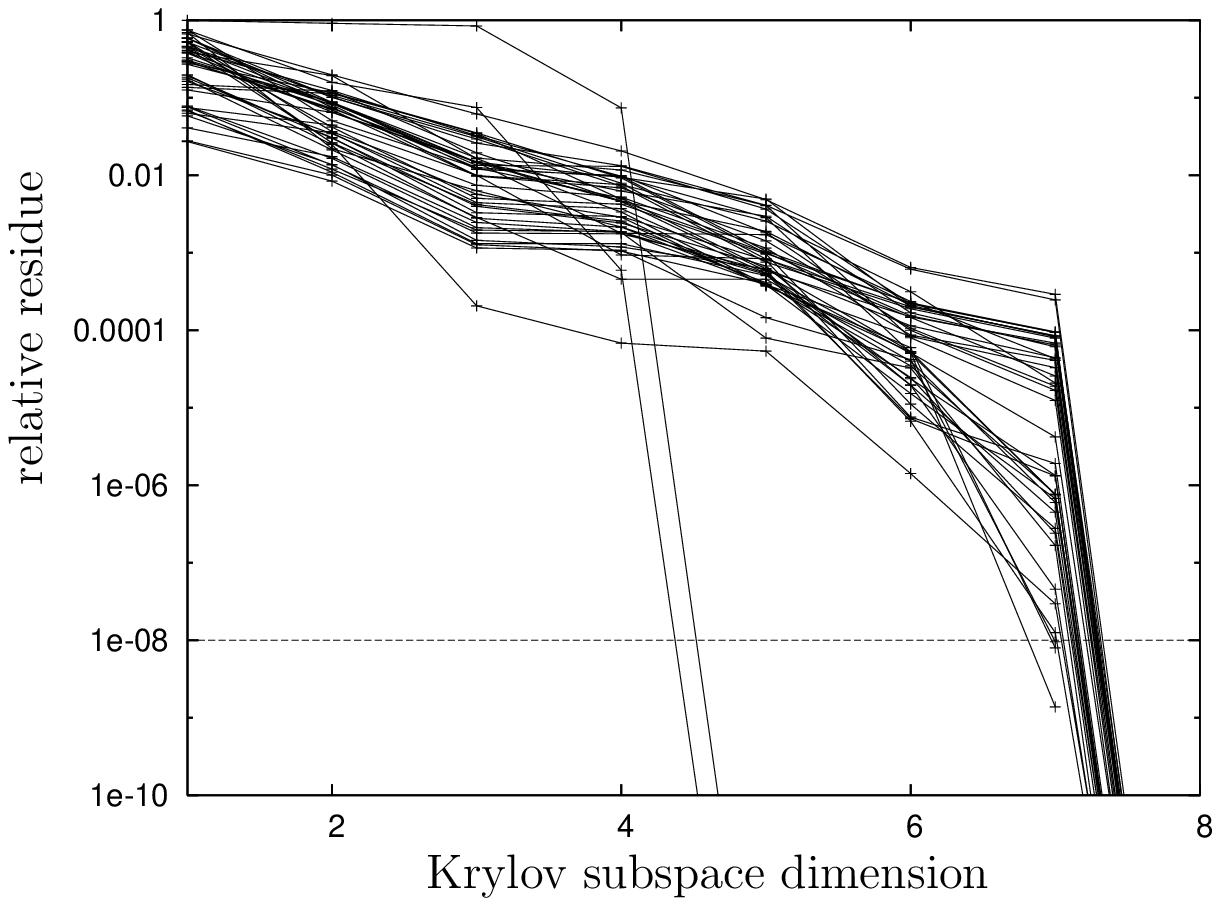,width=180pt}}
\end{picture}
\caption{Convergence results for the computations in plane Couette
  flow. Left: the residue $\|\mathbf{F}(\mathbf{z})\|$ of the correction step versus the number of
Newton-Krylov iterations. The dotted line denotes quadratic
convergence. Iterations were stopped when the residue was less than
$10^{-7}$. Right: the GMRES residue, normalised by
$\|\mathbf{F}(\mathbf{z})\|$, as a function of the Krylov subspace
dimension. These results were obtained for the continuation shown in
Fig. \ref{cont_curve_cou}, with three shooting intervals. By
proposition \ref{max_dim} the maximal subspace dimension is $8$. }
\label{conv_cou}
\end{figure}

Finally, we measured the wall time of the computation as a function of
the number of CPUs employed, as shown in Fig. \ref{wtime_cou}. In this computation we used five shooting
intervals of approximately equal length. With five CPUs active the
wall time was one fifth of the wall time with a single CPU to within
5\%.
The data communicated by {\sf MPI} comprises only a small number of
vectors of length $n$, mounting to less than a megabyte, and the code
can easily be run on an ordinary multi core computer.
\begin{figure}
\begin{center}
\includegraphics[width=0.7\textwidth]{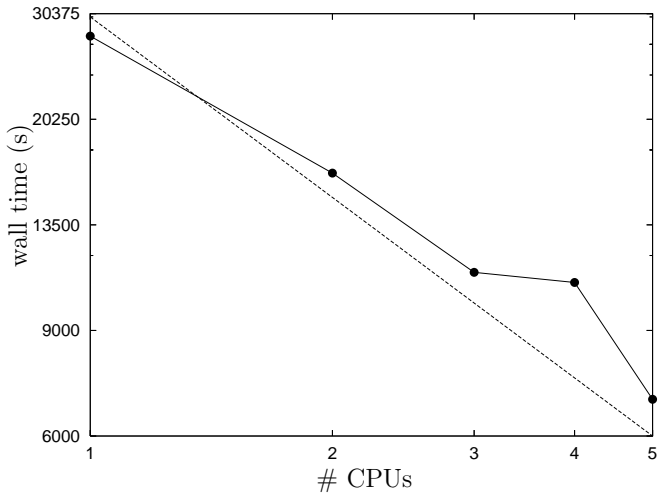}
\end{center}
\caption{Wall time versus the number of CPUs for a representative part
  of the continuation presented in Fig. \ref{cont_curve_cou}. The dots
  have been connected by lines for easy comparison. In each
  computation five shooting intervals of approximately equal
  integration time were used. The computation time
  decreases approximately by factors of $3/5$, $2/5$, $2/5$ and $1/5$
  when increasing the number of CPU's. This behaviour confirms that
  nearly all time is spent in time stepping the (linearised)
  Navier-Stokes equations in parallel and the parallelisation is
  nearly 100\% efficient. The dashed line denotes linear scaling.
 }
\label{wtime_cou}
\end{figure}

\section{Conclusion}

We have presented an efficient and flexible method for computing 2D
invariant manifolds of dynamical systems with any number of degrees of
freedom. This method is based on the orbit continuation algorithm
\cite[Sec. 3]{kraus}.The main issue in this computation is the sensitive
dependence on initial conditions. Our algorithm deals with this
problem by the use of multiple shooting. By we controlling the
integration time on each shooting interval we ensure the computation
is stable and well-conditioned. By choosing different boundary
conditions on each interval we can select different parts of the
manifold to compute. The time integrations of the dynamical system and
its linearisation on different shooting intervals can be efficiently
executed in parallel.

The multiple shooting orbit continuation leads to linear systems of a
size which grows as the product of the number of degrees of freedom of
the dynamical system and the number of shooting intervals. We have
implemented GMRES \cite{saad} as a linear solver.
Remarkably, Proposition \ref{max_dim} states that the maximal number of GMRES
iterations for each Newton update step is linear in the number of
shooting intervals but does not depend on the number of degrees of
freedom. In practice, this means that if we compute the same piece of
manifold several times with an increasing numbers of shooting
intervals and parallel processes, there is no guarantee that the
computation time will decrease. Normally, however, we will be
computing as large a piece as we can with the minimal number of
shooting intervals that guarantees convergence. 
For this approach, our convergence result implies that the computation
time will depend only on the time required by the time-stepping and 
on the condition of the linear system, not its size.

Both example computations in this paper concerned 2D unstable
manifolds of periodic solutions to strongly dissipative
systems. However, the algorithm has wider applicability. For stable
manifolds one can reverse the direction of time, or reverse the role
of the left and right boundary conditions. For (un)stable manifolds of
equilibria, the left boundary condition can be formulated in terms of
the two (un)stable eigenvectors. Moreover, the result on convergence
of GMRES iterations does not rely on any assumptions on the properties
of the linearised system. Thus, we expect the Newton-Krylov orbit
continuation algorithm to be equally suitable for the computation of
manifolds in general high-dimensional dynamical systems such as
networks of chaotic oscillators.

\section*{Acknowledgements}

The authors would like to thank Sebius Doedel for many useful discussions and Sadayoshi Toh
for making available his time-stepping code for plane Couette flow.

\bibliographystyle{siam}
\bibliography{mfcum}
\end{document}